\documentclass[a4paper,12pt]{amsproc}

\usepackage[utf8]{inputenc}

\usepackage{amsmath}
\usepackage{amssymb}
\usepackage{mathrsfs}
\usepackage[dvips]{graphicx}

\newtheorem{definition}{Definition}[section]
\newtheorem{theorem}{Theorem}[section]

\newtheorem{claim}{Claim}[section]

\newtheorem{proposition}{Proposition}[section]

\newtheorem{fact}{Fact}[section]
\newtheorem{remark}{Remark}[section]

\def\bbp{\mathbb{P}}

\def\bbr{\mathbb{R}} 
\def\bbs{\mathbb{S}}

\def\ca{\mathcal{A}}
\def\cb{\mathcal{B}}

\def\cd{\mathcal{D}}
\def\cf{\mathcal{F}}

\def\ch{\mathcal{H}}
\def\cm{\mathcal{M}}
\def\cn{\mathcal{N}}
\def\cp{\mathcal{P}}
\def\cw{\mathcal{W}}

\def\c{\mathfrak{c}}

\def\iff{\longleftrightarrow}
\def\then{\longrightarrow}

\def\force{\Vdash}
\def\hold{\vDash}

\def\restricted{\upharpoonright}

\def\iff{\longleftrightarrow}
\def\then{\longrightarrow}

\def\<{\langle}
\def\>{\rangle}

\title{Note on $s_0$ nonmeasurable unions}
\author{Robert Rałowski}\thanks{The work has been partially financed by NCN means granted by decision DEC-2011/01/B/ST1/01439.}
\address{ Institute of 
         Mathematics and Computer Science, 
         Wroc{\l}aw University of Technology, Wybrze\.ze Wyspia\'n\-skie\-go 27, 
         50-370 Wroc{\l}aw, Poland.}

\email[Robert Ra{\l}owski]{robert.ralowski@pwr.wroc.pl}

\begin{document}

\begin{abstract} In this note we consider an arbitrary families of sets of $s_0$ ideal introduced by Marczewski-Szpilrajn. We show that in any uncountable Polish space $X$ and under some combinatorial and set theoretical assumptions ($cov(s_0)=\c$ for example), that for any family $\ca\subseteq s_0$ with $\bigcup\ca =X$, we can find a some subfamily $\ca'\subseteq\ca$ such that the union $\bigcup\ca'$ is not $s$-measurable. We have shown a consistency of the $cov(s_0)=\omega_1<\c$ and existence a partition of the size $\omega_1$ $\ca\in [s_0]^{\omega}$ of the real line $\bbr$, such that there exists a subfamily $\ca'\subseteq\ca$ for which $\bigcup\ca'$ is $s$-nonmeasurable. We also showed that it is relatively consistent with ZFC theory that $\omega_1<\c$ and existence of m.a.d. family $\ca$ such that $\bigcup\ca$ is $s$-nonmeasurable in Cantor space $2^\omega$ or Baire space $\omega^\omega$. The consistency of $a<cov(s_0)$ and $cov(s_0)<a$ is proved also.
\end{abstract}
\maketitle

\section{Definitions}

Let $X$ be any uncountable Polish space and consider an arbitrary $\sigma$-ideal $I\subset P(X)$ then let us recall the cardinal coefficients
\begin{itemize}
\item $non(I)=\min \{ |F|:\; F\subseteq X\land F\notin I\},$
\item $add(I)=\min \{ |\ca|:\; \ca\subseteq I\land \bigcup\ca\notin I \},$
\item $cov(I)=\min \{ |\ca|:\; \ca\subseteq I\land \bigcup\ca = X\},$
\item $cov_h(I)=\min \{ |\ca|:\; \ca\subseteq I\land (\exists P\in Perf(X))\; P\subseteq \bigcup\ca\}$.
\end{itemize}

Marczewski introduced the notion $s$ measurability and $s_0$-ideal now we recall these definitions, see \cite{Marczewski}.

\begin{definition}[$s_0$ Marczewski ideal] Let $X$ be any fixed uncountable Polish space. Then we say that $A\in P(X)$ is in $s_0$ Marczewski ideal iff
$$
(\forall P\in Perf(X))(\exists Q\in Perf(X))\; Q\subseteq P\land Q\cap A=\emptyset.
$$
\end{definition}

\begin{definition}[$s$ measurable set] Let $X$ be any fixed uncountable Polish space. Then we say that $A\in P(X)$ is {\bf $s$-measurable} iff
$$
(\forall P\in Perf(X))(\exists Q\in Perf(X))\; (Q\subseteq A)\lor (Q\subseteq P\land Q\cap A=\emptyset).
$$
Moreover, the set $A\in P(X)$ is {\bf completely $s$-nonmeasurable} if
$$
(\forall P\in Perf(X))\; P\cap A\ne\emptyset\land P\cap A^c\ne \emptyset,
$$
where $A^c$ denotes completion of the set A in space $X$.
\end{definition}

It is well known by Judach Miller Shelah see \cite{JMS} and Repicki see \cite{Rep} that $add(s_0)\le cov(s_0)\le cof(\c)\le non(s_0)=\c<cof(s_0)\le 2^\c$.

\begin{definition} We say that the family $\ca\subset P(X)$ is {\bf $s$-summable} in the uncountable Polish space $X$ iff for every subfamily $\cb\subseteq \ca$ $\bigcup \cb$ is $s$-measurable set.
\end{definition}

\begin{definition} Let $X$ be any uncountable Polish space and let us consider a cardinal $\kappa$. We say that the family $\ca\subseteq P(X)$ is {\bf $\kappa$-point family} iff for any $x\in X$ $|\{ A\in\ca:\ x\in A\}|<\kappa$. 
\end{definition}

\section{Results on nonmeasurability respect to $s_0$-ideal}

We start this section with the simple observations in the sequel Propositions.

\begin{proposition}\label{summable_family} Let $X$ be any uncountable Polish space and let $\c$ be regular. Then there exists a $s$-summable family $\ca\subseteq s_0$ s.t. $\bigcup \ca=X$.
\end{proposition}
\begin{proof} Let us enumerate $X=\{ x_\xi:\; \xi<\c\}$ and let us define $A_\alpha=\{ x_\xi:\; \xi<\alpha\}\in s_0$ for any $\alpha<\c$. The family $\ca=\{ A_\alpha:\; \alpha<\c\}$ fulfills the assertion of the Theorem.
\end{proof}

\begin{fact} Every Luzin and Sierpiński set is in $s_0$.
\end{fact}

\begin{proposition} Let $\c$ be regular and let us assume that $\ca\subset \{ A\in P(X):\; X\text{ is a Łuzin set}\}$ s.t. $\bigcup \ca \notin s_0$ and is $\c$-point family then there exists $\cb\subseteq\ca$ s.t. $\bigcup\cb$ is not $s$-measurable set.
\end{proposition}
\begin{proof} First we show that $add(\ca,s_0)=\min \{ |\cb|:\; \cb\subseteq\ca\land \bigcup\cb\notin s_0 \}=~\c$. Let observe that for any perfect set $P\in Perf(X)$ there exists meager perfect subset of $P$. Then we can consider the family of the perfect sets which are meager only. Consider the subfamily $\cb\in [\ca]^{<\c}$ of the size less than $\c$. Observe that any member $A\in \cb$ has the countable intersection with the $P$ but we can decompose $P$ onto $\c$ many disjoint perfect sets $Q_\xi:\xi<\c$ then we can find $\c$ many $Q_\xi$ such that $Q_\xi\cap \bigcup\cb=\empty set$ what witness that $\bigcup\cb$ is in $s_0$.

We can assume that $\bigcup\ca$ is $s$-measurable, then one can find $P_0\in Perf(X)\cap\cm$ with $P_0\subseteq \bigcup\ca$. Now let us enumerate  $Perf(X\cap P_0)=\{ P_\xi:\xi<\c\}$. Then by transfinite induction we can build te following sequence:
$$
\< (A_\xi,d_\xi)\in\ca\times X:\;\; \xi<\c\>\text{ with the following properties:}
$$

\begin{itemize}
 \item $A_\xi\cap P_\xi\ne\emptyset$ and $d_\xi\in P_\xi$ for every $\xi<\c$ and
 \item $\{ d_\xi: \xi<\c\} \cap \bigcup\{ A_\xi:\; \xi<\c\}=\emptyset$.
\end{itemize}
In the $\alpha$ step of induction by the fact that $add(\ca,s_0)$ the union $\bigcup \{ A\in\ca:\; (\exists \xi<\alpha)\; d_\xi\in A\}$ is $s_0$ and then one can find an $A\in\ca$ such that $A\cap P_\alpha \ne\emptyset$ and $A\cup \{ d_\xi:\;\xi<\alpha\}=\emptyset$. After this, find $d\in P_\alpha\setminus (A\cup\bigcup\{ A_\xi:\xi<\alpha \})$ and put $A_\alpha=A$ and $d_\alpha=d$.
\end{proof}
\begin{remark} This is also true if we replace a Łuzin set by a Sierpiński set.
\end{remark}

Now we show that $\c$-point family assumption can not be committed.
\begin{proposition}[CH] There exists $s$-summable family $\ca$ of the Łuzin sets such that $\bigcup\ca=\bbr$.
\end{proposition}
\begin{proof} Let us enumerate all reals $x_\xi:\xi<\omega_1$ and let us consider family of $\omega_1$ many Łuzin sets $\{ C_\xi:\xi<\omega_1\}$ such that for every $\xi<\omega_1$ $x_\xi\in C_\xi$. Now let $A_\xi=\bigcup\{ C_\eta:\; \eta<\xi\}$ for every $\xi<\omega_1$ and let $\ca=\{ A_\xi:\xi<\omega_1\}$. Of course family $\ca$ consists of Łuzin sets and if $\eta<\xi<\omega_1$ then $A_\eta\subseteq A_\xi$ and $\bigcup\ca=\bbr$,
\end{proof}

It is well known that it is consistent with ZFC theory that $add(s_0)=\c=\omega_2$ see \cite{JMS} or \cite{Rep} for example. We give another way to show consistency of this fact with existence large cardinals.

\begin{theorem} Under existence a supercompact cardinal there exists a generic extension in which $\c=\omega_2$ and $cov_h(s_0)=\c$.
\end{theorem}
\begin{proof} For simplicity let assume that $X=\omega^\omega$ and let $P$, be any compact perfect subset of $X$. It is well known that under existence a supercompact cardinal there exists model in which PFA is holds (with $\c=\omega_2$). Let us consider a proper forcing $\bbp\subseteq \omega^{<\omega}$ which consists of all perfect trees s.t. $[p]\subseteq P$ for any $p\in\bbp$ with $\le=\subseteq$ relation. Now, let us choose any family $\ca\subseteq s_0$ of size $\omega_1$ and consider the family $\cd=\{ D_A:\; A\in\ca\}$ of dense subsets of $\bbp$ defined as follows
$$
(\forall A\in\ca)(\forall p\in \bbp)\; p\in D_A\iff p\cap A=\emptyset.
$$
Then by PFA there exists a $\cd$-generic set $G\subset\bbp$ (i.e. $G\cap D_A\ne\emptyset$ for every $A\in\ca$). Let us observe that for any finite family $G_0=\{ q_0,\ldots, q_{n-1}\}\in [G]^{<\omega}$ there exists condition $p\in G$ which is under any condition from $G_0$. Then the our generic family $G$ of forcing condition has a finite intersection property and then $\bigcap G\ne\emptyset$ which is a Sacks real in fact. Finally we see that $\bigcap G\subseteq P\setminus \bigcup\ca\ne\emptyset$.
\end{proof}

\begin{theorem}\label{pfa_result} If PFA holds and $\c$ is regular then for any $\c$-point family $\ca\subseteq s_0$ of the subsets of the Polish space $X$ with $\bigcup\ca\notin s_0$. Then there exists a subfamily $\cb\subseteq \ca$ s.t. $\bigcup\cb$ is not $s$-measurable set.
\end{theorem}
\begin{proof} We prove this Theorem by transfinite induction. We can assume that $\bigcup\ca$ is $s$-measurable, then one can find $P_0\in Perf(X)\cap\cm$ with $P_0\subseteq \bigcup\ca$. Now let us enumerate  $Perf(X\cap P_0)=\{ P_\xi:\xi<\c\}$. Then by transfinite induction we can build te following sequence:
$$
\< (A_\xi,d_\xi)\in\ca\times X:\;\; \xi<\c\>\text{ with the following properties:}
$$

\begin{itemize}
 \item $A_\xi\cap P_\xi\ne\emptyset$ and $d_\xi\in P_\xi$ for every $\xi<\c$ and
 \item $\{ d_\xi: \xi<\c\} \cap \bigcup\{ A_\xi:\; \xi<\c\}=\emptyset$.
\end{itemize}
In the $\alpha$ step of induction by the fact that $cov_h(s_0)=\c$ the union $\bigcup \{ A\in\ca:\; (\exists \xi<\alpha)\; d_\xi\in A\}$ doesn't cover $P_\alpha$ and then one can find an $A\in\ca$ such that $A\cap P_\alpha \ne\emptyset$ and $A\cup \{ d_\xi:\;\xi<\alpha\}=\emptyset$. After this, find $d\in P_\alpha\setminus (A\cup\bigcup\{ A_\xi:\xi<\alpha \})$ and put $A_\alpha=A$ and $d_\alpha=d$.

\end{proof}

The above Theorem uses the fact that $cov_h(s_0)=\c$. Here we show Theorem where we have $cov_h(s_0)<\c$.

\begin{theorem} It is relatively consistent that there exists a partition $\ca\in [s_0]^{\omega_1}$ of the real line for which there exists a subfamily $\cb\subseteq\ca$ such that $\bigcup\cb$ is not $s$-measurable and $\omega_1<\c$.
\end{theorem}
\begin{proof} Let us start from ground model $V$ in which $GCH$ is hold. Let us iterate with finite support of the Cohen forcing $C$, $\aleph_{\omega_1}$ many times and denote this forcing notion as $\bbp$. Of course or forcing notion is c.c.c. then every new real appears in the middle of 
iteration. Now, consider the following family $\{ V_\beta\cap \bbr^{V^\bbp}:\beta<\aleph_{\omega_1}\}$ of the real line $\bbr^{V^\bbp}$ where $V_0=V$. 
For the reader convenience we prove the following well known Claim.

\begin{claim}\label{new_branch} For any forcing notion $\bbp$ which adds a new real $c$ and any perfect tree $T\in \cp(2^{<\omega})\cap V$ in the ground model $V$ there exists a perfect subtree $S_c\subseteq T$ such that every branch of $S_c$ is a new i.e. $[S_c]\subseteq V^\bbp\setminus V$.
\end{claim}
\begin{proof} Let $T\subseteq 2^{<\omega}$ be any perfect tree in the ground model $V$. Let $c\in 2^\omega$ be any new real added by forcing $\bbp$. For any $n\in\omega$ let us define $Split_n(T)$ as follows:
$$
\{ s\in T:s^\smallfrown 0,s^\smallfrown 1\in T\land (\exists t\subseteq s)t\in Split_{n-1}(T)\land |s|\text{ is smallest}\}.
$$
Now let $S_c$ consists only those nodes of $T$ such that if $s\in Split_n(T)$ and $n=2k+1$ is odd then $s^\smallfrown i\in S$ iff $i=c(k)$. Let $x\in[S_c]$ be any branch of $S_c$. The tree $S_c$ is a perfect tree of course. We show that $x\in 2^\omega\cap (V^\bbp\setminus V)$. If not then let us consider the sequence such that for every $k\in\omega$
$$
y(k)=i\iff (\exists n\in\omega )\;\;x\upharpoonright n\in Split_{2k+1}(T)\land i=x(k)
$$
which is a set in the ground model but $y=c$ which is impossible because $c\notin V$, contradiction. 
\end{proof}

In each stage of the iteration (let say $\alpha$) the new real is added (Cohen real for example) $c_\alpha\in V\cap\bbr$ and then by above Claim \ref{new_branch} in every perfect tree $T$ from $V_\alpha$ we can find a perfect subtree $S\subseteq T$ for which every branch $b\in [S]$ is a new real $b\in V_{\alpha+1}\setminus V_\alpha$. This argument shows that $\ca=\{ \bbr^V\}\cup \{ (V_{\alpha+1}\setminus V_\alpha)\cap\bbr^{V^\bbp}:\alpha<\aleph_{\omega_1}\}\subset s_0$ forms a partition of the real line in the generic extension. Now let us consider the following family:
$$
\cb=\{ (V_{\alpha+1}\setminus V_\alpha)\cap\bbr^{V^\bbp}: \alpha<\aleph_{\omega_1}\land \alpha \text{ is odd }\}.
$$
To finish the proof we show that $A=\bigcup\cb$ is not $s$-measurable. To do this we show the little stronger statement that for every perfect tree $T$ and any perfect subtree $S\subseteq T$ such that $A\cap [S]\ne\emptyset$ and $A^c\cap [S]\ne\emptyset$. By c.c.c. of $\bbp$ every $S\in V_\alpha$ appears in a some $\alpha<\aleph_{\omega_1}$ stage of the iteration. Then we can find a perfect subtrees $S_2\subseteq S_1\subseteq S$ for which which $S_1\in V_{\alpha+1}$, $[S_1]\cap \bbr^{V^\bbp}\cap V_\alpha=\emptyset$ but $[S_1]\cap\bbr^{V^\bbp}\cap V_{\alpha+1}\ne\emptyset$ and $S_2\in V_{\alpha+2}$, $[S_2]\cap\bbr^{V^\bbp}\cap V_{\alpha+1}=\emptyset$ but $[S_2]\cap\bbr^{V^\bbp} \cap V_{\alpha+2}\ne\emptyset$.

The standard argument shows that $\c=\aleph_{\omega_1}^V=\aleph_{\omega_1}^{V^\bbp}$. The proof is finished.
\end{proof}

\section{Big point $s_0$-families and their $s$-nonmeasurability}
In contrast with the previous section where we proved the consistency result in Theorem \ref{pfa_result} which deals the families with small point property, we consider the big-point families of sets $\ca$ from ideal $s_0$ in the following meaning:

If $X$ be any Polish space then any family $\ca\subset\cp(X)$ is a {\bf big-point family} iff
$$
\{ x\in X:\;\; |\{ A\in\ca: x\in A|\}<\c\}\in s_0.
$$
The family constructed in the proof of the Proposition \ref{summable_family} is a big-point family. But in some additional assumption we can prove the following Theorem.
\begin{theorem}\label{big-point} If $cov(s_0)=\c$ and $\c$ is regular cardinal then if $\ca\subseteq s_0$ is a big-point family of Marczewski null subsets of real line such that
$$
(\forall x,y\in \bbr)\; x\ne y\then |\{ A\in\ca:\; x,y\in A\}|<\c.
$$
Then there exists a subfamily $\ca'\subseteq \ca$ such that $\bigcup\ca'$ is completely $s$-nonmeasurable.
\end{theorem}
\begin{proof} We prove this Theorem using transfinite induction. Then let us enumerate the set of all perfect subsets $Perf=\{ P_\xi:\xi<\c\}$ of the fixed uncountable Polish space $X$. We will build recursively a sequence of the length $\c$, 
$
\{(A_\xi,d_\xi)\in \ca\times P_\xi:\xi<\c\}
$
such that for every $\xi<\c$ we have
\begin{itemize}
    \item $A_\xi\cap P_\xi\ne\emptyset$ and
    \item $\{ d_\eta:\eta<\xi\} \cap \bigcup\{ A_\eta:\eta<\xi\}=\emptyset$.
\end{itemize}
Assume that we have sequence with the above properties of the fixed length $\xi<\c$. Choose any point $x\in P_\xi\setminus \bigcup_{\eta<\xi}\{ d_\eta\}$ then by assumption in our Theorem there exists a some $A\in\ca$ such that $x\in X$ and $d_\eta\notin A$ for every $\eta<\xi$. Now by an assumption that $cov(s_0)=\c$ choose any $d\in P\setminus (\bigcup\{ A_\eta:\eta<\xi\}\cup A)$. Finally set $A_\xi=A$ and $d_\xi=d$. Then $\{(A_\eta,d_\eta)\in \ca\times P_\eta:\eta\le\xi\}$ fulfills the analogous bullets as above. Then by transfinite induction Theorem we can construct a sequence with the length of $\c$ with an above properties. These properties shows that a family $\ca'=\{ A_\xi:\xi<\c\}$ fulfills an assertion of this Theorem.
\end{proof}

Let observe that the assertion of the above Theorem is true if PFA is hold.

\section{MAD $s_0$-families and their $s$-nonmeasurability}

We start this section with the definition of {\bf a.d.}-family i.e any family of sets $\ca\subseteq [\omega]^\omega$ is {\bf a.d.}-family on $\omega$ if
$$
(\forall a,b\in\ca)\; a\ne b\then a\cap b\in [\omega]^{<\omega}.
$$
The two reals $f,g\in\omega^\omega$ in Baire space are eventually different {\bf e.d.} iff $f\cap g$ is finite subset of $\omega\times\omega$. Then let us observe that {\bf e.d.} family $\ca\subseteq\omega^\omega$ is an {\bf a.d.} family on $\omega\times\omega$. For this reason we will call the eventually different family as almost disjoint {\bf a.d.} family also. Maximal almost adjoint (or eventually different) families respect inclusion are called a {\bf m.a.d.} families.

We know that union over every {\bf a.d.} family or eventually different functions family is a meager in the Cantor and Baire space. But it is well known that it is provable in ZFC that there exists a m.a.d family which contain a some uncountable perfect set which makes this set $s_0$ positive one. It is a natural question that it is true that there exists a m.a.d. family which union forms $s$-nonmeasurable set. From the other side it is well known that consistent is an existence a m.a.d. family $\ca$ with the cardinality is less than $\c$ see \cite{Kunen} book for example. Moreover, $non(s_0)=\c$ then the union of such a family $\bigcup\ca$ is in $s_0$.

\subsection{Consistency of $s$-nonmeasurable m.a.d. family}
We show the consistency of the existing a $m.a.d.$ family $\ca$ such that $\bigcup \ca$ is $s$-nonmeasurable in the
Baire space for example. We have the following Theorem.

\begin{theorem}\label{main_mad} It is relatively consistent with ZFC theory that There exists a {\bf m.a.d.} family of functions $\ca\subseteq \omega^\omega$ such that $\bigcup \ca$ is not $s$-measurable.
\end{theorem}
\begin{proof} Let us consider the ground model $V$ of $GCH$. To do we first choose any perfect tree $T\subseteq
\omega^{<\omega}$ in $V$ such that $[T]$ forms an $a.d.$ family. Now, let us define a forcing notion $(Q,\le)$ as
follows:
$p=(x_p,s_p^g,s_p^b,\cf_p,\ch_p)\in Q$ iff
\begin{itemize}
 \item $x_p\in \omega^{<\omega}$ and 
 \item $s_p^g,s_p^b\in [\omega^{|x_p|}]^{<\omega}$ and 
 \item $\cf_p\in [ST]^{<\omega}$ and
 \item $\ch_p\in [\omega^\omega]^{<\omega}$,
\end{itemize}
and fulfills the following conditions:
\begin{enumerate}
    \item $|s_p^g|=|s_p^b|=|\cf_p|$ and
    \item if $\cf=\{ T_k:k\in n\}\land s_p^g=\{ s_k: k\in n\} \then s_k\in T_k$ for every $k\in n$,
    \item if $\cf=\{ T_k:k\in n\} \land s_p^b=\{ s_k: k\in n\} \then s_k\in T_k$ for every $k\in n$.
    \item $(\forall S\in \cf) s_p^g\cap S\ne\emptyset\land s_p^b\cap S\ne\emptyset$
\end{enumerate}
Here $ST$ stands for the family of all perfect subtrees of the tree $T$.

The order is defined as follows: for every $p=(x_p,s_p^g,s_p^b,\cf_p),q=(x_q,s_q^g,s_q^b,\cf_q)\in Q$ we have $p\le q$ iff
\begin{enumerate}
    \item $x_q\subset x_p \land \cf_q\subseteq \cf_p \land \ch_q\subseteq \ch_p$ and
    \item $(\forall s\in s_q^g)(\exists t\in s_p^g)(s\subseteq t)$ and
    \item $(\forall s\in s_q^b)(\exists t\in s_p^b)(s\subseteq t)$ and
    \item $(\forall s\in s_q^g)(\forall t\in s_p^g)(s\subseteq t\then x_p\cap s=x_p\cap t)$ and
    \item $(\forall h\in \ch_q)(x_p\cap h=x_q\cap h)$ and
    \item $(\forall h\in\ch_q)(\forall s\in s_q^g)(\forall t\in s_p^g)(s\subseteq t\then s\cap h=t\cap h)$.
\end{enumerate}

From the definition of the our forcing notion we have the following Claims.
\begin{claim}\label{ccc_mad} $Q$ is c.c.c.
\end{claim}
\begin{proof} The proof goes in traditional way. First choose any two conditions $p,q\in Q$ with $x_p= x_q=x, s_p^g=s_q^g=s^g,s_p^b=s_q^b=s^b$.  The following forcing condition $r=(x,s^g,s^b,\cf_p\cup \cf_q,\ch_p\cup \ch_q)\in Q$  is a common extension (has more information) than $p$ and $q$.
Now consider the uncountable set $\cw\in [Q]^{\omega_1}$ of forcing conditions of the poset $Q$. Then there exists an uncountable subset $\cw_0\in [\cw]^{\omega_1}$ s.t. each member of $\cw_0$ has the same first coordinate. Then there exists uncountable $\cw_1\in[\cw_0]^{\omega_1}$ with the same the second coordinate and then we can find $\cw_2\in [\cw_1]^{\omega_1}$ which is also uncountable and the third coordinate is the same for all conditions from $\cw_2$. Then by above remark all conditions in $\cw_2$ are comparable.
\end{proof}

\begin{claim}\label{s^g_mad} Let $G\subseteq Q$ generic filter over $V$. Then in $V[G]$ the family $\{ x_G\}\cup [\{ s: (\exists p\in G)(\exists t\in s_p^g) s\subseteq t \}]$ is {\bf a.d.}, where $x_G=\bigcup\{ x_p: p\in G \}$ is a generic real.
\end{claim}
\begin{proof} Choose any element $y$ belongs to the set $[\bigcup s_p^g:\; p\in G]$ then for every $n\in\omega$ there exists a condition $p\in G$ and $s\in s_p^g$ such that $y\restriction_n\subseteq s$. Then we can find decreasing sequence $(p_n)_{n\in\omega}\in G^\omega$such that for every $n\in\omega$ $y\restriction_n\subseteq s$ for some $s\in s_{p_n}^g$ (here $p_n=(x_p,s_p^g,s_p^b,\cf_p)\in G$). Then we have $x_{p_n}\cap s_{p_n}=x_{p_n}\cap s_{p_0}$ and then
$$
x_{p_n}\cap y\restriction_n\subseteq x_{p_n}\cap s_{p_n}=x_{p_n}\cap S_{p_0}\subseteq s_{p_0}.
$$
Then finally $x_G\cap y\subseteq s_{p_0}$. From the other side for each condition $p\in Q$ we have $s_p^g\subseteq T$ but the formula "tree has {\bf a.d.} branches only" i.e.
$$
(\forall x)(\forall y)(\forall n\in\omega) (x\ne y\land x\restricted_n\in T\land y\restricted_n\in T)\then (x\cap y\in T)
$$
is $\prod_1^1$ formula then is absolute between transitive ZF models of the set theory. Then $\{ x_G\}\cup [\{ s: (\exists p\in G)(\exists t\in s_p^g) s\subseteq t \}]$ forms {\bf a.d.} family.
\end{proof}

For a reader convenience we give the proof of the next Claim.
\begin{claim}\label{s^b_mad} Let $G\subseteq Q$ generic filter over $V$. Then in $V[G]$ the family $\{ x_G\}\cup [\{ s: (\exists p\in G)(\exists t\in s_p^b) s\subseteq t \}]$ is not $a.d.$, where $x_G=\bigcup\{ x_p: p\in G \}$ is a generic real.
\end{claim}
\begin{proof} It is easy to show that the $\{ D_n: n\in\omega\}$ is family of dense sets in $Q$, where
$$
D_n=\{ p\in Q: |x_p|\ge n\land (\forall s\in s_p^b)(\exists m>n) x_p(m)=s(m)\}
$$
for each $n\in\omega$.
\end{proof}

\begin{claim}\label{x_G_mad} If $G\subseteq Q$ is $Q$-generic over $V$. Then for any perfect subtree $S\in ST\cap V$ of $T$ the $[S]\cap [\{ \bigcup s_p^g: p\in G\}]\ne\emptyset$ and $[S]\cap [\{ s: (\exists p\in G)(\exists t\in s_p^g) s\subseteq t \}]\ne\emptyset$ is hold in the generic extension $V[G]$. Moreover, in $V[G]$, for any old real $h\in \omega^\omega\cap V$, $x_G\cap h$ is finite.
\end{claim}
\begin{proof} Choose any perfect subtree $S$ of the tree $T$ from the ground model $V$. Observe that $D^g_n=\{ p\in Q:\; (\forall s\in s_p^g)(n< |s|)\}$ and $D^b_n=\{ p\in Q:\; (\forall s\in s_p^b)(n< |s|)\}$ are dense for every $n\in\omega$ and $F_S=\{ p\in Q:\; S\in \cf_p\}$ is also dense in $Q$.
Then by induction we can to build the decreasing sequence $(r_n)_{n\in\omega}\in Q^\omega$, sequence
$(p_n)_{n\in\omega}\in Q^\omega$ in the poset $Q$ and $(t_n)_{n\in\omega}\in ({\omega^{<\omega})}^\omega$ such that
\begin{itemize}
 \item $G\ni r_n\le p_n\in D^g_n$ for every $n\in\omega$,
 \item $t_n\subseteq t_{n+1}\in S$ and $t_n\in s_{r_n}^g$ for every $n\in\omega$.
\end{itemize}
Then the real $t=\bigcup\{ t_n\in \omega\}$ whiteness the fact that $t\in [S]\cup [\{s: (\exists p\in G)(\exists t\in s_p^g) s\subseteq t \}]$. The same argument provide us to the existence a real $s\in \omega^\omega$ such that $s\in [S]\cup [\{ s: (\exists p\in G)(\exists t\in s_p^g) s\subseteq t\}]$.

Consider the following families $\{ D_n: n\in\omega\}$, $\{ E_h:\; h\in \omega^\omega\}$ of subsets of $Q$ defined as follows:
$$
D_n=\{ p\in Q:\; n< |x_p|\},\;
E_{h}=\{ p\in Q:\; h\in\ch_p\}
$$
It is easy to see that $E_h$ is dense in $Q$ for each $h\in 2^\omega\cap V$. Then choose any condition $p\in G\cap E_h$ and any positive integer $n\in\omega$ then there exists a condition $q\in G\cap D_n$ then one can find an extension $p_n$ of $p,r$ which is in generic filter $G$. Then $x_{p_n}\cap h=x_p\cap h\subseteq x_p$ for every $n\in\omega$ then $x_G\cap h\subseteq x_p$ which is finite one.
\end{proof}

Now to prove our Theorem we use the following Claim.
\begin{claim}\label{newreals_mad} In $V[G]$, for every real $h\in \omega^\omega\cap V$ in ground model and for every $y\in [\{ s:\; (\exists p\in G)(\exists t\in s_p^g)s\subseteq t \}]\; |y\cap h|$ is finite.
\end{claim}
\begin{proof} Let $G\subseteq Q$ be fixed generic filter over $V$. Choose any ground model real $h\in \omega^\omega\cap V$
and $y$ as is stated in the Claim. Then for every $n\in \omega$ there exists condition $p_n\in G$ such that
$y\restricted_n\subset s_n$ for some $s_n\in s_p^g$ (here $n\le |s|$ of course). As above the set $E_h=\{ p\in Q: h\in
\ch_p\}$ is dense in $Q$ then find a some condition $r\in G\cap E_h$. Now by induction find a  sequences decreasing
$(r_n)_{n\in\omega}\in Q^\omega$ and $(t_n)_{n\in\omega}\in (\omega^{<\omega})^\omega$ such that
\begin{itemize}
    \item $r\ge r_0$ and $G\ni r_n\le p_n$ for all $n\in\omega$ and
    \item $y\restricted_n\subseteq s_n \subseteq t_n$ and $t_n\in s_{r_n}^g$ for every $n\in\omega$ and
    \item $h\cap t_0=h\cap t_n$ and $y\restricted_n=t_n\restricted_n$ for any $n\in\omega$.
\end{itemize}
Then $y=\bigcup\{ t_n\restricted_n:\; n\in\omega\}$ and $h\cap t_0\subseteq h\cap y\restricted_n$ for large enough $n\in\omega$ ($n\ge |t_0|$) and then $h\cap y=h\cap t_0$ which is finite.
\end{proof}
Now let us consider any cardinal $\kappa$ with a uncountable cofinality and finite support iteration $((P_\alpha:\alpha\le \kappa), (\dot{Q}_\beta:\beta<\kappa))$ such that for every we have $\beta<\kappa$ $\force_{P_\beta}\; \dot{Q}_\beta = \hat{Q}$.
Assume that $G_\beta = \{ p\in P_\beta: i_{\beta \kappa}(p)\in G\}$ where $G\supset P_\kappa$ generic filter over $V$ and $\beta<\kappa$. Then there exists $H\subseteq \dot{Q}_{\beta_{G_\beta}}$ generic over universe $V[G_\beta]$ such that $G_{\beta+1}=G_\beta*H$. Now let us define the following family $\ca_\beta = \{ x_{G_{\beta+1}}\}\cup [\{ s: (\exists p\in G_{\beta+1})(\exists t\in s_p^g) s\subseteq t\}]$ and then $\ca=\bigcup\{ \ca_\beta:\beta<\kappa\}$. In $V[G]$ we show that $\ca$ forms {\bf a.d.} and for every $\cb$ {\bf m.a.d.} family containing $\ca$ $S\in ST$ such that $\bigcup \cb\cap [S]\ne\emptyset$ and $(\bigcup\cb)^c\cap [S]\ne\emptyset$ what shows that there exists a {\bf m.a.d.} family which is not $s$-measurable. First of all let observe 
that $Z=\bigcup_{\beta<\kappa} [\{ s: (\exists p\in G_{\beta+1})(\exists t\in s_p^g) s\subseteq t\}]\subseteq [T]$ but $T\in V$ is almost disjoint tree, but this property is $\prod_1^1$ and by Shoenfield Theorem is absolute between transfinite ZFC models. Then the last set $Z$ consists of almost disjoint reals. Let $X=\{ x_{G_\beta}: \beta<\kappa\}$. Let $\alpha<\beta$ then by the Claim \ref{x_G_mad} $x_{G_\alpha}\cap x_{G_\beta}$ is finite. 
Moreover, if $s\in [\{ s: (\exists p\in G_\alpha)(\exists t\in s_p^g) s\subseteq t\}]$ then once again by the Claim \ref{x_G_mad} $x_{G_\beta}\cap s$ is finite. If $\beta<\alpha$ then by the Claim \ref{newreals_mad} $x_{G_\beta}\cap s$ is finite for every $s\in [\{ s: (\exists p\in G_\alpha)(\exists t\in s_p^g) s\subseteq t\}]$. By Claim \ref{s^g_mad} $x_{G_\beta}\cap s$ is also finite for every $s\in[\{ s: (\exists p\in G_\alpha)(\exists t\in s_p^g) s\subseteq t\}]$ and $\beta<\kappa$. This shows that $\ca$ is {\bf a.d.} family.

Now we show that any extension $\cb \supseteq\ca$ to {\bf m.a.d.} family is $s$-nonmeasurable. Choose in $V[G]$ any perfect subtree $S\in ST$ of the our tree $T$. Let $\dot{S}\in V^{P_\kappa}$ be a nice name for a $S$ which is countable. But $P_\kappa$ is an finite support iteration of $c.c.c.$ forcings see Claim \ref{ccc_mad} then there exists $\beta<\kappa$ and $P_\alpha$-name $\dot{S}'$ such that $\dot{S}'_{G_\beta}=S=\dot{S}_G$. Then by the Claim \ref{s^g_mad} $[S]\cap [\{ s: (\exists p\in G_\beta)(\exists t\in s_p^g) s\subseteq t\}]\ne\emptyset$ and by the Claim \ref{s^b_mad} there exists $s\in [\{ s: (\exists p\in G_\alpha)(\exists t\in s_p^b) s\subseteq t\}]\cap [S]$ for which $s\cap x_{G_\beta}$ is infinite so $(\bigcup\cb)^c\cap [S]\ne\emptyset$.
\end{proof} 

\subsection{Consistency of $cov(s_0)<\mathfrak{a}$} In Yorioka paper \cite{Yorioka} is proved that starting in ground model with CH the $\omega_2$ finite support iteration of Hehler forcing gives the model in which $cov(s_0)=\omega_1$ but this forcing adds $\omega_2$ Hechler reals and then any {\bf a.d.} family $\ca\subseteq\omega^\omega$ of size $\omega_1$ can not be a maximal because the Hechler forcing is c.c.c. . Then in this model we have $\omega_1=cov(s_0)<\mathfrak{a}=\omega_2$.

\subsection{Consistency of $\mathfrak{a} < cov(s_0)$}

We start this section with definition which we use in proof of consistency mentioned in the title above.
\begin{definition}[Strong m.a.d. family] A family $\ca\subseteq\omega^\omega$ is called {\bf strongly m.a.d.} if for every countable set $\{ f_i:i\in\omega\}$ of reals avoiding a family $\ca$ there exists a $g\in\ca$ such that $|h\cap f_i|=\omega$. Here $f\in\omega^\omega$ is {\bf avoids} $\ca$ iff for every finite set $\cb\in [\ca]^{<\omega}$ we have $|f\setminus\bigcup \cb|=\omega$.
\end{definition}
Kanstermans (see \cite{Kaster}) showed that under MA the {\bf strongly m.a.d.} family exists.

To show consistency result mentioned above we will use an existence oc the supercompact cardinal. We assume that reader is familiar with the theory of large cardinals. For details we refer to the chapters $20, 31$ of the classical handbook \cite{Jech}  or to the chapter $12$, written by Cummings in \cite{bigbook}.

\begin{definition}[Normal ultrafilter] Let $\kappa$ and $\lambda\ge\gamma$ are uncountable cardinals then we say that $\kappa$-complete ultrafilter $U$ on $[\lambda]^{<\kappa}$ is normal if for every $f:[\lambda]^{<\kappa}\to \lambda$ with $\{ x\in [\lambda]^{<\kappa}: f(x)\in x\}\in U$ implies that there exists a some $\gamma\in\lambda$ such that $\{ x: f(x)=\gamma\}\in U$. Such ultrafilter we can call a normal measure on $[\lambda]^{<\kappa}$.
\end{definition}
For a fixed normal ultrafilter $U$ on $[\lambda]^{<\kappa}$ let $j_U:V\to Ult(V,U)$ be elementary embedding defined by $j_U(x)=[c_x]\in Ult(V,U)$ where $[c_x]$ be equivalence class of the constant  function $c_x(\alpha)=x$ for any $\alpha\in\lambda$.

It is well known that
\begin{theorem} Let $\kappa\le\lambda$ then there exists a normal measure on $[\lambda]^{<\kappa}$ iff there exists an elementary embedding $j:V\to M$ such that
\begin{enumerate}
 \item $(\forall \alpha<\kappa)\; j(\alpha)=\alpha$ and $\kappa<j(\kappa)$ and
 \item $M^\lambda\subseteq M$.
\end{enumerate}
\end{theorem}
If $\kappa$ fulfills the above two conditions then $\kappa$ is called $\lambda$-supercompact cardinal. We say that $\kappa$ is {\bf supercompact} if for every $\lambda\ge\kappa$ the $\kappa$ is a $\lambda$-supercompact cardinal.

To prove the consistency result that $a< cov(s_0)$ we will follow the proof of the consistency of the $PFA$. Then we need to have a Laver function in hands what is guaranteed by the existence of supercompact cardinal. Here we recall the famous Laver Theorem.

\begin{theorem}[Laver]\label{Laver} Let $\kappa$ be supercompact then there exists a function $f:\kappa\to V_\kappa$ (called a {\bf Laver function}) such that for every $\lambda\ge\kappa$ and $x\in H_{\lambda^+}$ there exist a supercompact measure $U$ on $[\lambda]^{<\kappa}$ and elementary embedding $j_U:V\to M$ such that $j_U(f)(\kappa)=x$ holds.
\end{theorem}

\begin{theorem} Under existence a supercompact cardinal it is relatively consistent with ZFC theory that $\omega_1=cof(\cn)=\mathfrak{a} < cov(s_0)=\omega_2=2^\omega$.
\end{theorem}
\begin{proof} Let $V\hold GCH$ and assume that $\kappa$ be a supercompact and $f:\kappa\to V_\kappa$ be a Laver function as above see Thm \ref{Laver}. Then let us consider a forcing notion $((P_\alpha:\alpha\le\kappa),(\dot{Q}_\alpha:\alpha<\kappa))$ with a countable support iteration such that for every $\alpha<\kappa$ $\dot{Q}_\alpha=f(\alpha)$ whenever $f(\alpha)\in \{ \dot{\bbs},\dot{Coll}(\omega_1,\alpha)\}\subseteq V^{P_\alpha}$ and $\dot{Q}_\alpha$ is a $P_\alpha$-name for a trivial forcing in the other case.

Let observe that $\{\alpha<\kappa: \dot{Q}_\alpha=\dot{\bbs}\}$ is unbounded in $\kappa$ and the same for $\dot{Coll}(\omega_1,\alpha)$ which implies that in generic extension $V[G]$ $\omega_2=\kappa$ and $\kappa=2^\omega$ is hold.

Now repeating the arguments in proof of the consistency of the PFA we show that in $V[G]$ for any family $\cd\in [\bbs]^{\omega_1}$ of dense sets of the Sacks forcing $\bbs$ there exists a $\cd$-filter $H\subseteq \bbs$.

In generic extension $V[G]$ now fix a family $\ca\in [s_0]^{<\kappa}$ and consider a family $\cd=\{ D_A:\; A\in\ca\}$ of dense subsets of Sacks forcing $\bbs$ defined as follows:
$$
D_A=\{ p\in\bbs: A\cap [p]=\emptyset\}.
$$
This is easy to see that $\bigcap\{ p\in H:(\exists A\in\ca) p\in H\cap D_A\}$ is nonempty set (consists a Sacks real) and is disjoint from $\bigcup\ca$. Moreover it is well known that the iteration with countable support preserves the Sacks property and thus $cof(\cn)=\omega_1$ see \cite{Miller-1981} or \cite{BaJud}, \cite{Shelah}.

$\bbp_\kappa$ is an countable support iteration of proper forcings then $\omega_1$ is not collapsed in generic extension. Moreover the mentioned $\bbp_\kappa$ is an iteration of the length $\kappa$ (which is regular) of forcings of the size less than $\kappa$ then $\bbp_\kappa$ is $\kappa$-c.c. forcing then $\kappa$ remains a cardinal in the generic universe $V[G]$. $\bbp_\kappa$ adds at least $\kappa$ new Sacks reals then in $V[G]$ $\omega_1<\kappa\le\c$. 

Moreover, $V\hold GCH$ and then for each $\alpha<\kappa$ we have $|P_\alpha|<\kappa$ and for every real $x$ in $V[G$ there exists name $\dot{x}$ for $x$ such that $supp(\dot{x})<\kappa$ what implies that $\c\le \kappa$. 

Let observe that for any $\alpha<\kappa$ and for any $\beta<\omega_1$ the set
$$
D_{\alpha,\beta}=\{ p\in Coll(\omega_1,\alpha): \beta\in range(p) \}
$$
is dense in forcing notion with a countable size of conditions $Coll(\omega,\kappa)$. Then for every $\alpha<\kappa$ there exists a $\cd_\alpha=\{ D_{\alpha,\beta}:\; \beta<\omega_1\}$-generic filter $G_\alpha$, producing a collapsing map from $\omega_1$ onto $\alpha$. Thus this implies that $\kappa=\omega_2^{V[G]}$ and by above $\c=\omega_2^{V[G]}$.

From the other side using preservation Theorem of strongly {m.a.d.}-families by countable support iteration of Sakcs forcing of arbitrary length see \cite{Raghavan} if $\ca\subseteq \omega^\omega$ is a strongly {\bf m.a.d.}-family then in $V[G]$ the family  $\ca\in [\omega^\omega]^{\omega_1}$ remains a {\bf m.a.d.}-family of size $\omega_1$.
The proof of this Theorem is completed.
\end{proof}

%
%


\begin{thebibliography}{1234567890}
\bibitem[Abraham]{Abraham} Abraham, U., Proper forcing,
\bibitem[Bart-Judah]{BaJud} Bartoszyński Tomek, Judah Haim, Set Theory, On the Structure of the Real Line, A K Peters Wellesley, Massachusetts, (1995).
\bibitem[BrYa]{BrYa} Brendle J., Yatabe S., Forcing indestructibility of MAD families, Annals of Pure and Applied Logic 132 (2005) 271-312.
\bibitem[Set Theory]{bigbook} Foreman M., Kanamori A. editors, Handbook of Set Theory, Springer, (2010).
\bibitem[Jech]{Jech} Jech T., Set theory, millenium edition, Springer Monographs in Mathematics, Springer-Verlag, (2003).
\bibitem[JuMiSh]{JMS} Judah H., Miller A., Shelah S., Sacks forcing, Laver forcing and Martin's Axiom, Archive for Math Logic 31 (1992) 145-161.
 \bibitem[Kaster]{Kaster} Kastermans, B., Very mad families, Advances in logic, Contemp. Math. , vol 425, Amer, Math. Soc., Providence, RI, 2007, pp. 105-112.
 \bibitem[Kunen]{Kunen} Kunen, K., Set Theory. An Introduction to Independence Proofs, North Holland, Amsterdamm, New York, Oxford 1980. 
\bibitem[Marcz]{Marczewski} Marczewski (Szpilrajn) E., Sur une classe de fonctions de W. Sierpiński et la classe correspondante d’ensembles, Fund. Math. 24 (1935), 17–34.
\bibitem[Miller]{Miller} Miller, A., A MAD Q-set Fundamenta Mathematicae, 178(2003), 271-281.
\bibitem[Miller 1981]{Miller-1981} Miller A., Some properties of measure and category, Transactions of American Mathematical Society, 266(1) (1981), pp. 93-114.
\bibitem[Raghavan]{Raghavan} Raghavan, D. Madness and Set Theory, Phd thesis (2008).
\bibitem[Repicki]{Rep} Repicki, M., Perfect sets and collapsing continuum, Comment. Math. Univ. Carolin. 44,2 (2003) 315–327.
 \bibitem[Sokup]{Sokup} Roitman J., Soukup L,, Luzin and anti-Luzin almost disjoint families, Fund. Math., 158(1998) , 51--67.
 \bibitem[Shelah]{Shelah} Shelah S., Proper and Improper Forcing, Springer, (1998).
 \bibitem[Yorioka]{Yorioka} Yorioka, T., Forcings with the countable chain condition and the covering number of the Marczewski ideal, Arch. Math. Logic 42 (2003), 695–710.
\end{thebibliography}
\end{document}